\documentclass{amsart}

\usepackage{hyperref}

\usepackage{amssymb}
\usepackage{amsmath}
\usepackage{amsthm}
\usepackage{enumerate}
\usepackage{graphicx}

\theoremstyle{plain}

\makeatletter
\newtheorem*{rep@theorem}{\rep@title}
\newcommand{\newreptheorem}[2]{%
\newenvironment{rep#1}[1]{%
 \def\rep@title{#2 \ref{##1}}%
 \begin{rep@theorem}}%
 {\end{rep@theorem}}}
\makeatother

\newtheorem{theorem}{Theorem}[section]
\newreptheorem{theorem}{Theorem}
\newtheorem{proposition}[theorem]{Proposition}

\newtheorem{lemma}[theorem]{Lemma}

\newtheorem{question}[theorem]{Question}

\theoremstyle{definition}
\newtheorem{definition}[theorem]{Definition}

\theoremstyle{remark}
\newtheorem*{remark}{Remark}

\newtheorem*{acknowledgments}{Acknowledgments}

\newcommand{\RiemannSphere}{\widehat{\mathbb{C}}}

\renewcommand {\tilde} {\widetilde}

\begin{document}

\title{Removability and non-injectivity of conformal welding}

\date{March 2, 2017}

\author[M. Younsi]{Malik Younsi}

\address{Department of Mathematics, University of Washington, Seattle, WA 98195-4350, United States.}
\email{malik.younsi@gmail.com}


\begin{abstract}
We construct a (non-removable) Jordan curve $\Gamma$ and a non-M\"{o}bius homeomorphism of the Riemann sphere which is conformal on the complement of $\Gamma$ and maps the curve $\Gamma$ onto itself. The curve is flexible in the sense of Bishop and may be taken to have zero area. The existence of such curves and conformal homeomorphisms is closely related to the non-injectivity of conformal welding.

\end{abstract}

\maketitle

\section{Introduction}

Let $\mathbb{D}$ be the open unit disk, let $\mathbb{D}^{*}:= \RiemannSphere \setminus \overline{\mathbb{D}}$ be the complement of the closed unit disk in the Riemann sphere $\RiemannSphere$, and let $\mathbb{T}:= \partial \mathbb{D}$ be the unit circle. Given a Jordan curve $\Gamma$, let $f:\mathbb{D} \to \Omega$ and $g:\mathbb{D}^{*} \to \Omega^{*}$ be conformal maps onto the bounded and unbounded complementary components of $\Gamma$ respectively. Then $f$ and $g$ extend to homeomorphisms on the closure of their respective domains, so that $h_\Gamma:=g^{-1} \circ f : \mathbb{T} \to \mathbb{T}$ defines an orientation-preserving homeomorphism of the unit circle onto itself, called the \textit{conformal welding homeomorphism} of $\Gamma$.

Note that $h_\Gamma$ is uniquely determined by $\Gamma$ up to pre- and post-composition by automorphisms of the unit disk. Moreover, if $T$ is a M\"{o}bius transformation, then $\Gamma$ and $T(\Gamma)$ have the same conformal welding homeomorphism, so that the \textit{conformal welding correspondence}
$$\mathcal{W}: [\Gamma] \mapsto [h_\Gamma]$$
is well-defined, from the family of Jordan curves, modulo M\"{o}bius equivalence, to the family of orientation-preserving homeomorphisms of the circle, modulo pre- and post-composition by automorphisms of the disk.

This correspondence between Jordan curves and circle homeomorphisms has appeared over the years to be of central importance in a wide variety of areas of mathematics and applications, such as Teichm\"{u}ller theory, Kleinian groups, computer vision and numerical pattern recognition (\cite{EKS},\cite{SHM}), and so forth. For more information on the applications of conformal welding, the interested reader may consult the survey article \cite{HAM2}. We also mention that recent years have witnessed a strong renewal of interest in conformal welding as other variants and generalizations have been introduced and developed, such as generalized conformal welding (\cite{BIS2}, \cite{HAM}), random conformal welding \cite{AST}, conformal welding for finitely connected regions \cite{MAR}, conformal welding of random surfaces \cite{SHE} and conformal laminations of trees (\cite{MARS}, \cite{ROH}), including applications to Shabat polynomials and Grothendieck's \textit{dessins d'enfants}. Conformal laminations are also related to the recent groundbreaking work of Miller and Sheffield on the relationship between the Brownian map and Liouville Quantum Gravity (\cite{MILS}, \cite{MILS2}, \cite{MILS3}).

It is well-known that the conformal welding correspondence $\mathcal{W}$ is not surjective; in other words, there are orientation-preserving homeomorphisms of the circle (even analytic everywhere except at one point) which are not conformal welding homeomorphisms. On the other hand, every quasisymmetric homeomorphism $h:\mathbb{T} \to \mathbb{T}$ is the conformal welding of some Jordan curve (in this case, a quasicircle). Here \textit{quasisymmetric} means that adjacent arcs $I,J \subset \mathbb{T}$ of equal length are mapped by $h$ onto arcs of comparable length :

$$M^{-1} \leq \frac{|h(I)|}{|h(J)|} \leq M.$$
The fact that every such $h$ is a conformal welding homeomorphism is usually referred to as the \textit{fundamental theorem of conformal welding} and was first proved by Pfluger \cite{PFL} in 1960. Another proof based on quasiconformal mappings was published shortly after by Lehto and Virtanen \cite{LEHV}. See also the papers of Bishop \cite{BIS2} and Schippers--Staubach \cite{SCS}. We also mention that other sufficient conditions for a given circle homeomorphism $h$ to be a welding were obtained by Lehto \cite{LEH} and Vainio \cite{VAI}, in terms of $h$ being sufficiently ``nice''. In Section \ref{sec1}, we recall a deep theorem of Bishop \cite{BIS2} saying that on the other hand, any ``wild'' enough $h$ is also the welding of some Jordan curve. Finding a complete characterization of conformal welding homeomorphisms is most likely a very difficult problem.

This paper, however, deals with the (non) injectivity of the welding correspondence $\mathcal{W}$. Recall that if $\Gamma$ is a Jordan curve, then $T(\Gamma)$ has the same welding homeomorphism $h_\Gamma$, for any M\"{o}bius transformation $T$. Are these the only curves with this property?

If $\Gamma$ and $\tilde{\Gamma}$ are two Jordan curves such that $h_\Gamma = h_{\tilde{\Gamma}}$, then there are conformal maps $f:\mathbb{D} \to \Omega$, $g:\mathbb{D}^{*} \to \Omega^{*}$, $\tilde{f} : \mathbb{D} \to \tilde{\Omega}$, $\tilde{g}: \mathbb{D}^{*} \to \tilde{\Omega}^{*}$ such that
$$g^{-1} \circ f = \tilde{g}^{-1} \circ \tilde{f}$$
on $\mathbb{T}$, i.e.
$$\tilde{g} \circ g^{-1} = \tilde{f} \circ f^{-1}$$
on $\Gamma$, where $\tilde{\Omega}, \tilde{\Omega}^{*}$ are the bounded and unbounded complementary components of the curve $\tilde{\Gamma}$. It follows that the map $F: \RiemannSphere \to \RiemannSphere$ defined by
\begin{displaymath}
F(z) = \left\{ \begin{array}{ll}
(\tilde{f} \circ f^{-1})(z) & \textrm{if $z \in \Omega \cup \Gamma$}\\
(\tilde{g} \circ g^{-1})(z) & \textrm{if $z \in \Omega^{*}$}\\
\end{array} \right.
\end{displaymath}
is a homeomorphism conformal on $\RiemannSphere \setminus \Gamma$ which maps the curve $\Gamma$ onto the curve $\tilde{\Gamma}$. This shows that if $\Gamma$ is conformally removable, then $\Gamma$ uniquely corresponds to its conformal welding homeomorphism $h_\Gamma$, up to M\"{o}bius equivalence.

\begin{definition}
We say that a compact set $E \subset \mathbb{C}$ is \textit{conformally removable} if every $F \in \operatorname{CH}(E)$ is M\"{o}bius, where $\operatorname{CH}(E)$ is the collection of homeomorphisms $F:\RiemannSphere \to \RiemannSphere$ which are conformal on $\RiemannSphere \setminus E$.
\end{definition}

Single points, smooth curves and more generally, sets of $\sigma$-finite length, are all conformally removable. On the other hand, the theory of quasiconformal mappings can be used to prove that sets of positive area are never conformally removable. The converse is well-known to be false, and there exist nonremovable sets (even Jordan curves) of Hausdorff dimension one (\cite{BIS}, \cite{KAU}) and removable sets of Hausdorff dimension two (\cite[Chapter V, Section 3.7]{LEHV2}). In fact, no geometric characterization of conformally removable sets is known. We also mention that this notion of removability appears naturally in the study of various important problems in Complex Analysis and related areas, such as Koebe's uniformization conjecture (\cite{HES}, \cite{YOU2}) and the MLC conjecture on the local connectivity of the Mandelbrot set (\cite{KAH}, \cite{KAL}, \cite{KAL2}). See also the survey article \cite{YOU} for more information.

Now, recall that if a Jordan curve $\Gamma$ is conformally removable, then $\Gamma$ uniquely corresponds to its conformal welding homeomorphism, modulo M\"{o}bius equivalence. The starting point of this paper is the following question.

\begin{question}
\label{ques1}
Does the converse hold? Namely, if $\Gamma$ is a non-removable Jordan curve, does there necessarily exist another curve having the same welding homeomorphism, but which is not a M\"{o}bius image of $\Gamma$?
\end{question}

Several papers in the literature seem to imply that the answer is trivially yes, either without proof or with the following argument :
\\

\textit{If $\Gamma$ is not removable, then there exists a non-M\"{o}bius $F \in \operatorname{CH}(\Gamma)$. But then the curve $F(\Gamma)$ has the same welding homeomorphism as $\Gamma$, and is not M\"{o}bius equivalent to it, since $F$ is not a M\"{o}bius transformation}
\\

See e.g. \cite[Lemma 2]{OIK}, \cite[Corollary II.2]{KNS}, \cite[Section 4]{HAM}, \cite[Corollary 1]{BIS3}, \cite[p.324--325]{BIS}, \cite[Section 3]{HAM2}, \cite[Remark 2]{BIS2}, \cite[Section 2.3]{AST}, \cite[Corollary 1.4]{LIR}).

Although it is true and easy to see that $\Gamma$ and $F(\Gamma)$ have the same welding homeomorphism, it is not clear at all that these two curves are not M\"{o}bius equivalent, since there could a priori exist a M\"{o}bius transformation $T$ such that $F(\Gamma)=T(\Gamma)$, even though $F$ itself is not M\"{o}bius. As far as we know, this remark first appeared in Maxime Fortier Bourque's Master's Thesis \cite{MFB}. The question of whether such $\Gamma, F$ and $T$ as above actually exist was, however, left open. In this paper, we answer that question in the affirmative.

\begin{theorem}
\label{mainthm}
There exists a Jordan curve $\Gamma$ and a non-M\"{o}bius homeomorphism $F: \RiemannSphere \to \RiemannSphere$ conformal on $\RiemannSphere \setminus \Gamma$ such that $F(\Gamma)=\Gamma$. Moreover, the curve $\Gamma$ may be taken to have zero area.
\end{theorem}

The construction is based on a result of Bishop \cite{BIS2} characterizing the conformal welding homeomorphisms of so-called flexible curves.

Theorem \ref{mainthm} shows that the above argument claiming to answer Question \ref{ques1} is in fact incorrect, and whether removability really characterizes injectivity of conformal welding remains unknown.

The remainder of the paper is organized as follows. In Section \ref{sec1}, we recall Bishop's characterization of the conformal welding homeomorphisms of flexible curves. Then, in Section \ref{sec2}, we use this result to prove Theorem \ref{mainthm}. In Section \ref{sec3}, we discuss the non-injectivity of conformal welding for curves of positive area. Finally, Section \ref{sec4} contains some open problems related to Question \ref{ques1}.

\section{Flexible curves and log-singular homeomorphisms}
\label{sec1}

We first need the definition of logarithmic capacity, following \cite{BIS2}. For $E \subset \mathbb{T}$ Borel, let $\mathcal{P}(E)$ denote the collection of all Borel probability measures on $E$.

\begin{definition}
Let $\mu \in \mathcal{P}(E)$.

\begin{enumerate}[\rm(i)]
\item The \textit{energy} of $\mu$, noted $I(\mu)$, is given by
$$I(\mu):= \int \int \log{\frac{2}{|z-w|}} \, d\mu(z) \, d\mu(w).$$

\item The \textit{logarithmic capacity} of $E$, noted $\operatorname{cap}(E)$, is defined as
$$\operatorname{cap}(E):= \frac{1}{\inf \{I(\mu): \mu \in \mathcal{P}(E) \} }.$$
\end{enumerate}
\end{definition}

It is well-known that logarithmic capacity is nonnegative, monotone and countably subadditive (see e.g. \cite{CAR}). We will also need the simple fact that bi-H\"{o}lder homeomorphisms of the circle preserve sets of zero capacity.

We can now define log-singular homeomorphisms.

\begin{definition}
Let $I,J$ be two subarcs of the unit circle. An orientation-preserving homeomorphism $h: I \to J$ is \textit{log-singular} if there exists a Borel set $E \subset I$ such that both $E$ and $h(I \setminus E)$ have zero logarithmic capacity.
\end{definition}

The following inductive construction of log-singular homeomorphisms was outlined in \cite[Remark 9]{BIS2}. We reproduce it here for the reader's convenience.

\begin{proposition}
\label{prop1}
Let $I,J$ be two subarcs of $\mathbb{T}$. Then there exists a log-singular homeomorphism $h:I \to J$.
\end{proposition}

\begin{proof}
Start with any orientation-preserving linear homeomorphism $h_1 : I \to J$. At the first step, divide $I$ into two subarcs, denoted by \textit{red} and \textit{blue} respectively, in such a way that the red subarc has small logarithmic capacity, say less than $2^{-1}$. Now, define a homeomorphism $h_2$ on $I$ which is linear on the red subarc and the blue subarc, and which satisfies $h_2(I)=h_1(I)=J$. We also construct $h_2$ such that it maps the blue subarc onto an arc of logarithmic capacity also less than $2^{-1}$.

Now, at the $n$-th step, suppose that $I$ has been divided into a finite number of arcs $\{I^{k,n}\}$, and that we have a homeomorphism $h_n:I \to J$ which is linear on each of those arcs. First, divide each $I^{k,n}$ into $n$ arcs of equal length $I_1^{k,n}, I_2^{k,n}, \dots, I_n^{k,n}$. Then, divide each $I_j^{k,n}$ into a red and a blue subarc, in such a way that the union of all the red subarcs has logarithmic capacity less than $2^{-n}$. Now, define a homeomorphism $h_{n+1}:I \to J$ which is linear on each of the red and blue subarcs, and which satisfies $h_{n+1}(I_j^{k,n}) = h_n(I_j^{k,n})$ for each $j,k$. We also construct $h_{n+1}$ so that the union of all the images of the blue subarcs under the map has logarithmic capacity less than $2^{-n}$.

It is not difficult to see that these maps $h_n$ converge to an orientation-preserving homeomorphism $h:I \to J$. Indeed, if $\epsilon>0$, then we can choose $N$ sufficiently large so that for $n \geq N$, the arcs $h_n(I_j^{k,n})$ all have length less than $\epsilon$. If $m \geq n$ and $x \in I$, then $x$ belongs to one of the arcs $I_j^{k,n}$ and $h_m(x)$ belongs to $h_n(I_j^{k,n})$ by construction, thus
$$|h_n(x)-h_m(x)| < \epsilon.$$
This shows that the sequence $(h_n)$ is uniformly Cauchy and therefore has a continuous limit $h:I \to J$, which has to be an orientation-preserving homeomorphism, by construction.

Finally, the map $h:I \to J$ is log-singular. Indeed, if $E$ is the set of points in $I$ which belong to infinitely many of the red subarcs, then
$$\operatorname{cap}(E) \leq \sum_{n \geq m} 2^{-n} \qquad (m \in \mathbb{N})$$
by the subadditivity of logarithmic capacity, so that $\operatorname{cap}(E) =0$. On the other hand, we have
$$h(\mathbb{T} \setminus E) \subset \bigcup_{m=1}^{\infty} \bigcap_{n \geq m} h_{n+1}(B_n),$$
where $B_n$ is the union of all the blue subarcs constructed at the $n$-th step. It follows that $h(\mathbb{T} \setminus E)$ is contained in a countable union of sets of zero logarithmic capacity, so that $\operatorname{cap}(h(\mathbb{T} \setminus E)) =0$, again by the subadditivity of logarithmic capacity.

\end{proof}

We will also need the following definition.

\begin{definition}
\label{def1}
A Jordan curve $\Gamma \subset \mathbb{C}$ is a \textit{flexible curve} if the following two conditions hold :

\begin{enumerate}[\rm(i)]
\item Given any Jordan curve $\tilde{\Gamma}$ and $\epsilon>0$, there exists $F \in \operatorname{CH}(\Gamma)$ such that
$$d(F(\Gamma),\tilde{\Gamma}) < \epsilon,$$
where $d$ is the Hausdorff distance.
\item Given points $z_1,z_2$ in each complementary component of $\Gamma$ and points $w_1,w_2$ in each complementary component of $\tilde{\Gamma}$, we can choose $F$ above so that $F(z_1)=w_1$ and $F(z_2)=w_2$.
\end{enumerate}
\end{definition}

One can think of a flexible curve $\Gamma$ as being ``highly''  non-removable, in the sense that $\operatorname{CH}(\Gamma)$ is very large. Examples of flexible curves with any Hausdorff dimension between 1 and 2 were constructed in \cite{BIS}.

There is a close relationship between flexible curves and log-singular homeomorphisms. Indeed, Bishop proved in \cite{BIS2} that an orientation-preserving homeomorphism $h: \mathbb{T} \to \mathbb{T}$ is log-singular if and only if it is the conformal welding of a flexible curve $\Gamma$. In particular, this implies that every such $h$ is the conformal welding of a dense family of curves, since if $F \in \operatorname{CH}(\Gamma)$ is as in Definition \ref{def1}, then $h$ is also the conformal welding homeomorphism of $F(\Gamma)$. We shall actually need the following stronger result, see \cite[Theorem 25]{BIS2}.

\begin{theorem}[Bishop]
\label{thmBis}
Let $h:\mathbb{T} \to \mathbb{T}$ be an orientation-preserving log-singular homeomorphism with $h(1)=1$, and let $f_0, g_0$ be two conformal maps of $\mathbb{D}$ and $\mathbb{D}^{*}$ respectively onto disjoint domains. Then for any $0<r<1$ and any $\epsilon>0$, there are conformal maps $f$ and $g$ of $\mathbb{D}$ and $\mathbb{D}^{*}$ onto the two complementary components of a Jordan curve $\Gamma$ satisfying the following conditions :

\begin{enumerate}[\rm(i)]
\item $h= g^{-1} \circ f$ on $\mathbb{T}$;
\item $|f(z)-f_0(z)| < \epsilon$ for all $z \in \mathbb{D}$ with $|z| \leq r$;
\item $|g(z)-g_0(z)| < \epsilon$ for all $z \in \mathbb{D}^{*}$ with $|z| \geq 1/r$;
\end{enumerate}

Moreover, the maps $f,g$ may be constructed such that $f(1)=g(1)=\infty$ and such that the curve $\Gamma$ has zero area.
\end{theorem}

\begin{remark}

Since the last part of Theorem \ref{thmBis} was not stated explicitly in \cite{BIS2}, let us briefly explain how it follows from the construction. First, since $h(1)=1$, condition $(i)$ implies that $f(1)=g(1)$. Composing $f$ and $g$ by a M\"{o}bius transformation $T$ if necessary, we can assume that $f(1)=g(1)=\infty$. Note that if $f$ and $g$ approximate $T^{-1} \circ f_0$ and $T^{-1} \circ g_0$ on compact subsets of $\mathbb{D}$ and $\mathbb{D}^{*}$ respectively, then $T \circ f$ and $T \circ g$ approximate $f_0$ and $g_0$.

Now, to see that the curve $\Gamma$ can be taken to have zero area, note that the main part of the proof in \cite{BIS2} is to construct \textit{quasiconformal} mappings $f:\mathbb{D} \to \Omega$ and $g:\mathbb{D} \to \Omega^*$ onto the complementary components of a Jordan curve $\Gamma$ satisfying conditions (i), (ii) and (iii), and having quasiconstants close to 1. These maps $f$ and $g$ are obtained as limits of quasiconformal mappings $f_n : \mathbb{D} \to \Omega_n$ and $g_n : \mathbb{D}^{*} \to \Omega_n^{*}$ onto smooth Jordan domains with disjoint closures, with $\infty \in \Omega_n^*$. By construction, the curve $\Gamma$ is contained in the topological annulus $A_n:=\RiemannSphere \setminus (\Omega_n \cup \Omega_n^{*})$, for each $n$. Moreover, the domains $\Omega_n$ and $\Omega_n^{*}$ are of the form $\Omega_n:=F_n(t\mathbb{D})$ and $\Omega_n^*:=G_n(\mathbb{D}^{*}/t)$, where $t<1$ is sufficiently close to 1 and $F_n, G_n$ are quasiconformal mappings of $\mathbb{D}, \mathbb{D}^*$ onto the complementary components of a smooth Jordan curve $\Gamma_n$. Clearly, by taking $t$ closer to $1$ if necessary, we can assume that the topological annulus $A_n$ has area as small as we want, say less than $2^{-n}$. Then the curve $\Gamma$ will have zero area.

Finally, the last part of the proof is to apply the measurable Riemann mapping theorem to replace $f$ and $g$ by conformal maps $\Phi \circ f$ and $\Phi \circ g$, where $\Phi$ is a quasiconformal mapping of the sphere. Since $\Phi$ preserves sets of area zero and since it can be assumed to fix $\infty$ by composing with a M\"{o}bius transformation, the new conformal maps can also be taken so that they send $1$ to $\infty$ and map $\mathbb{T}$ onto a curve of zero area.

\end{remark}

\section{Conformal homeomorphisms fixing a curve}
\label{sec2}

We can now prove Theorem \ref{mainthm}. Recall that we want to construct a Jordan curve $\Gamma$ and a non-M\"{o}bius homeomorphism $F: \RiemannSphere \to \RiemannSphere$ conformal on $\RiemannSphere \setminus \Gamma$ such that $F(\Gamma)=\Gamma$.

The idea of the construction is the following. Suppose that such a curve $\Gamma$ and such a non-M\"{o}bius homeomorphism $F \in \operatorname{CH}(\Gamma)$ exist, and suppose that $F$ preserves the orientation of $\Gamma$. Let $f : \mathbb{D} \to \Omega$ and $g : \mathbb{D}^{*} \to \Omega^{*}$ be conformal maps onto the two complementary components of the curve. Then $F \circ f$ and $F \circ g$ are also conformal maps of $\mathbb{D}$ and $\mathbb{D}^{*}$ onto $\Omega$ and $\Omega^{*}$ respectively, so that
$$F \circ f = f \circ \sigma$$
and
$$F \circ g = g \circ \tau$$
for some $\sigma, \tau \in \operatorname{Aut}(\mathbb{D})$. Note that neither $\sigma$ nor $\tau$ is the identity, since otherwise $F$ would also be the identity, contradicting the fact that it is not M\"{o}bius. Now, since $F$ is continuous on $\Gamma$, we must have
$$f \circ \sigma \circ f^{-1} = g \circ \tau \circ g^{-1}$$
there, which can be rewritten as
\begin{equation}
\label{functionalequation}
W \circ \sigma = \tau \circ W
\end{equation}
on $\mathbb{T}$, where $W:=h_\Gamma$ is the conformal welding of $\Gamma$. We thus obtain a functional equation for the welding of the curve.

The strategy now is to proceed backward. Start with an orientation-preserving homeomorphism $W : \mathbb{T} \to \mathbb{T}$ satisfying the functional equation (\ref{functionalequation}) for some $\sigma, \tau \in \operatorname{Aut}(\mathbb{D})$. Suppose in addition that we can construct $W$ so that it is the conformal welding homeomorphism of some Jordan curve $\Gamma$, i.e. $W = g^{-1} \circ f$ where $f$ and $g$ are conformal maps of $\mathbb{D}$ and $\mathbb{D}^{*}$ respectively onto the two complementary components $\Omega$, $\Omega^{*}$ of $\gamma$. Then we can define a map $F$ conformal on both sides of $\gamma$ by

\begin{displaymath}
F(z) = \left\{ \begin{array}{ll}
(f \circ \sigma \circ f^{-1})(z) & \textrm{if $z \in \Omega$}\\
(g \circ \tau \circ g^{-1})(z) & \textrm{if $z \in \Omega^{*}$},\\
\end{array} \right.
\end{displaymath}
and the fact that $W=g^{-1} \circ f$ satisfies Equation (\ref{functionalequation}) implies that $F$ extends to a homeomorphism of the whole sphere. Clearly, the map $F \in \operatorname{CH}(\Gamma)$ fixes the curve $\Gamma$, and all that remains to prove is that we can choose $\sigma, \tau$ and $W$ such that $F$ is not a M\"{o}bius transformation.

The main difficulty here is that if we choose $W$ to be sufficiently nice, e.g. quasisymmetric, so that it is a conformal welding homeomorphism, then the curve $\Gamma$ will be removable and $F$ will necessarily be M\"{o}bius. In order to circumvent this difficulty, a promising approach is to construct log-singular homeomorphic solutions of the functional equation (\ref{functionalequation}).

\begin{lemma}
\label{lemlog}
Let $a,b>0$, and let $\phi$ be a M\"{o}bius transformation mapping the upper half-plane $\mathbb{H}$ onto the open unit disk $\mathbb{D}$ with $\phi(\infty)=1$, say
$$\phi(z) := \frac{z-i}{z+i}.$$
Define $\tilde{\sigma}, \tilde{\tau} \in \operatorname{Aut}(\mathbb{H})$ by $\tilde{\sigma}(z):=z+a$ and $\tilde{\tau}(z):=z+b$, and set
$$\sigma := \phi \circ \tilde{\sigma} \circ \phi^{-1}$$
and
$$\tau := \phi \circ \tilde{\tau} \circ \phi^{-1},$$
so that $\sigma, \tau \in \operatorname{Aut}(\mathbb{D})$. Then there exists a log-singular orientation-preserving homeomorphism $W: \mathbb{T} \to \mathbb{T}$ which satisfies
$$W \circ \sigma = \tau \circ W$$
\end{lemma}

\begin{proof}
Let $I_0:= \phi([0,a))$ and $J_0:=\phi([0,b))$, and let $W: I_0 \to J_0$ be a log-singular orientation-preserving homeomorphism, as in Proposition \ref{prop1}. For $n \in \mathbb{Z}$, set $I_n := \sigma^n(I_0)$ and $J_n := \tau^n(J_0)$. Note that the subarcs $I_n$ are pairwise disjoint and that
$$\bigcup_{n \in \mathbb{Z}} I_n = \mathbb{T} \setminus \{1\},$$
and similarly for the $J_n$'s. Now, extend $W$ to all of $\mathbb{T}$ by setting
$$W(z) := (\tau^n \circ W \circ \sigma^{-n})(z) \qquad (z \in I_n)$$
and $W(1):=1$. Clearly, the map $W : \mathbb{T} \to \mathbb{T}$ thereby obtained is an orientation-preserving homeomorphism. Moreover, if $z \in I_n$ for some $n$, then $\sigma(z) \in I_{n+1}$, so that
\begin{eqnarray*}
W(\sigma(z)) &=& (\tau^{n+1} \circ W \circ \sigma^{-(n+1)})(\sigma(z))\\
&=& (\tau \circ (\tau^n \circ W \circ \sigma^{-n}))(z)\\
&=& (\tau \circ W)(z),
\end{eqnarray*}
which shows that $W$ satisfies Equation (\ref{functionalequation}).

It remains to prove that $W: \mathbb{T} \to \mathbb{T}$ satisfies the log-singular condition. Let $E_0 \subset I_0$ such that $\operatorname{cap}(E_0)=0$ and $\operatorname{cap}(W(I_0 \setminus E_0))=0$. For $n \in \mathbb{Z}$, let $E_n:=\sigma^n(E_0) \subset I_n$. Note that $\operatorname{cap}(E_n)=0$ for all $n$, since bi-H\"{o}lder homeomorphisms preserve sets zero logarithmic capacity. It follows that $E:=\cup_n E_n$ also has capacity zero, by subadditivity. Finally, we have
\begin{eqnarray*}
W(\mathbb{T} \setminus E) &=& \bigcup_{n \in \mathbb{Z}} W(I_n \setminus E_n) \cup \{1\}\\
&=& \bigcup_{n \in \mathbb{Z}} W(\sigma^n(I_0 \setminus E_0))  \cup \{1\}\\
&=& \bigcup_{n \in \mathbb{Z}} \tau^n(W(I_0 \setminus E_0))  \cup \{1\},
\end{eqnarray*}
which shows that $\operatorname{cap}(W(\mathbb{T} \setminus E))=0$, again by subadditivity and the fact that bi-H\"{o}lder homeomorphisms preserve sets of zero logarithmic capacity.

This completes the proof of the lemma.

\end{proof}

We can now proceed with the proof of Theorem \ref{mainthm}.

\begin{proof}
Let $a,b >0$ with $a \neq b$, and let $W: \mathbb{T} \to \mathbb{T}$, $\sigma$ and $\tau$ be as in Lemma \ref{lemlog}, so that
$$W \circ \sigma = \tau \circ W.$$
Also, let $0<\epsilon<|a-b|/4$, let $z_1, z_2$ be points in the upper half-plane $\mathbb{H}$ and lower half-plane $\mathbb{H}^{*}$ respectively, and let $K_1 \subset \mathbb{H}$ and $K_2 \subset \mathbb{H}^{*}$ be compact sets such that $z_1, z_1+a \in K_1$ and $z_2, z_2+b \in K_2$. Lastly, take $0<r<1$ sufficiently close to $1$ so that $\phi(K_1) \subset \mathbb{D}(0,r)$ and $\phi(K_2) \subset \mathbb{C} \setminus \mathbb{D}(0,1/r)$, where $\phi$ is the M\"{o}bius transformation of Lemma \ref{lemlog}.

By Theorem \ref{thmBis}, we can write $W:= g^{-1} \circ f$, where $f$ and $g$ are conformal maps of $\mathbb{D}$ and $\mathbb{D}^{*}$ onto the two complementary components $\Omega$ and $\Omega^{*}$ of a Jordan curve $\Gamma$, such that $f(1)=g(1)=\infty$,

\begin{equation}
\label{ineq1}
|f(z)-\phi^{-1}(z)| < \epsilon \qquad (|z| \leq r)
\end{equation}
and
\begin{equation}
\label{ineq2}
|g(z)-\phi^{-1}(z)| < \epsilon \qquad (|z| \geq 1/r).
\end{equation}
In particular, the above inequalities hold for $z \in \phi(K_1)$ and $z \in \phi(K_2)$ respectively.

Now, define a map $F$ by
\begin{displaymath}
F(z) = \left\{ \begin{array}{ll}
(f \circ \sigma \circ f^{-1})(z) & \textrm{if $z \in \Omega$}\\
(g \circ \tau \circ g^{-1})(z) & \textrm{if $z \in \Omega^{*}$},
\end{array} \right.
\end{displaymath}

so that $F$ is conformal on $\RiemannSphere \setminus \Gamma$. Also, the equation $W \circ \sigma = \tau \circ W$ on $\mathbb{T}$ is equivalent to
$$f \circ \sigma \circ f^{-1} = g \circ \tau \circ g^{-1}$$
on $\Gamma$, so that $F$ extends to a homeomorphism of $\RiemannSphere$. Clearly, this map satisfies $F(\Gamma)=\Gamma$.

It remains to prove that $F$ is not a M\"{o}bius transformation. Suppose, in order to obtain a contradiction, that $F$ is M\"{o}bius. First, note that $F(\infty)=\infty$, so that $F$ has to be linear, say $F(z)=cz+d$. Now, the map $F$ can be rewritten as
\begin{displaymath}
F(z) = \left\{ \begin{array}{ll}
(\tilde{f} \circ \tilde{\sigma} \circ \tilde{f}^{-1})(z) & \textrm{if $z \in \Omega$}\\
(\tilde{g} \circ \tilde{\tau} \circ \tilde{g}^{-1})(z) & \textrm{if $z \in \Omega^{*}$},\\
\end{array} \right.
\end{displaymath}
where $\tilde{f}:=f \circ \phi : \mathbb{H} \to \Omega$, $\tilde{g}:= g \circ \phi : \mathbb{H}^{*} \to \Omega^{*}$, and $\tilde{\sigma}(z):=z+a$, $\tilde{\tau}(z):=z+b$ are as in Lemma \ref{lemlog}. It is easy to see from this that $F$ has only one fixed point, at infinity, so that $c=1$.

Now, note that for $z \in \mathbb{H}$, we have
$$f(\phi(z))+d=\tilde{f}(z) + d = F(\tilde{f}(z)) = (\tilde{f} \circ \tilde{\sigma})(z) = \tilde{f}(z+a) = f(\phi(z+a)),$$
so that in particular,
$$d-a = f(\phi(z_1+a)) - f(\phi(z_1))-a = f(\phi(z_1+a)) - (z_1+a) - f(\phi(z_1))+z_1.$$
By Inequality (\ref{ineq1}) with $z$ replaced by $\phi(z_1), \phi(z_1+a) \in \phi(K_1)$, we get
$$|d-a| \leq \epsilon + \epsilon = 2\epsilon.$$
Similarly, using the formula for $F$ on $\Omega^{*}$ and Inequality (\ref{ineq2}), we get
$$|d-b| \leq 2\epsilon,$$
and combining the two inequalities yields
$$|a-b| \leq |a-d| + |d-b| \leq 4\epsilon,$$
which contradicts the choice of $\epsilon$. It follows that $F$ is not a M\"{o}bius transformation.

Finally, we can take $\Gamma$ to have zero area, by Theorem \ref{thmBis}.

This completes the proof of Theorem \ref{mainthm}.

\end{proof}

\section{Non-injectivity of conformal welding for curves of positive area}
\label{sec3}

In this section, we mention that the argument described in the introduction, although incorrect in general, can nonetheless be made to work in the case of curves with positive area.

\begin{theorem}
\label{thm2}
If $\Gamma$ is a Jordan curve with positive area, then there is another curve having the same welding homeomorphism, but which is not a M\"{o}bius image of $\Gamma$.
\end{theorem}

The idea of the proof is quite simple. Since $\Gamma$ has positive area, it is in particular not removable, so there is a non-M\"{o}bius homeomorphism $F : \RiemannSphere \to \RiemannSphere$ which is conformal off $\Gamma$. As already mentioned, the curve $F(\Gamma)$ has the same welding as $\Gamma$, thus is a good candidate for the curve we want to construct. Unfortunately, as we saw in Section \ref{sec2}, it may happen that this curve is a M\"{o}bius image of the original one, even though $F$ itself is not M\"{o}bius.

However, in the positive-area case, it is easy to see using the measurable Riemann mapping theorem that the collection of non-M\"{o}bius elements of $\operatorname{CH}(\Gamma)$ is infinite-dimensional, in some sense. A dimension argument relying on Ahlfors-Bers and Brouwer's Invariance of Domain can then be applied to conclude that there must be at least one non-M\"{o}bius $F \in \operatorname{CH}(\Gamma)$ such that $F(\Gamma) \neq T(\Gamma)$ for every M\"{o}bius transformation $T$.

As far as we know, Theorem \ref{thm2} was first stated by Katznelson, Nag and Sullivan in \cite{KNS}. See \cite[Theorem 4.22]{MFB} for a complete and detailed proof. It would be very interesting to find a more constructive proof though.

\section{Concluding remarks}
\label{sec4}

In view of Theorem \ref{thm2}, Question \ref{ques1} can be reduced to the following.

\begin{question}
\label{ques2}
If $\Gamma$ is a non-removable Jordan curve with zero area, does there necessarily exist another curve having the same conformal welding homeomorphism, but which is not a M\"{o}bius image of $\Gamma$?
\end{question}

As observed in Section \ref{sec3}, the size of $\operatorname{CH}(\Gamma)$ may be relevant here.

\begin{question}
If $\Gamma$ is a non-removable Jordan curve with zero area, is the collection of non-M\"{o}bius elements of $\operatorname{CH}(\Gamma)$ necessarily large, or infinite-dimensional, in some sense?
\end{question}

If $\Gamma$ is non-removable, then there exists at least one non-M\"{o}bius homeomorphism $F:\RiemannSphere \to \RiemannSphere$ conformal off $\Gamma$, but as far as we know, it is still open in general whether there must exist another non-M\"{o}bius element of $\operatorname{CH}(\Gamma)$ which is not of the form $T \circ F$, for $T$ M\"{o}bius.

Finally, we conclude by mentioning that a positive answer to Question \ref{ques2} would follow if one could prove that there always exists a non-M\"{o}bius homeomorphism of $\RiemannSphere$ conformal off $\Gamma$ which maps $\Gamma$ onto a curve of positive area.

\acknowledgments{The author thanks Chris Bishop and Don Marshall for helpful discussions.}

\bibliographystyle{amsplain}

\end{document}